\documentclass[12pt]{article}

\usepackage{url}

\usepackage{pdfscreen}
\usepackage{hyperref}

\usepackage{mathptmx}
\usepackage{latexsym}
\usepackage{amssymb}
\usepackage{amsthm}
\usepackage{amsmath}
\usepackage{fancyhdr}
\usepackage{mathptmx}
\usepackage{latexsym}
\usepackage{amssymb}

\usepackage{authblk}
\newtheorem{theorem}{Theorem}
\newtheorem{lemma}{Lemma}
\newtheorem{corollary}{Corollary}

\begin{document}
\title{Gr\"unbaum coloring and its generalization\\ to arbitrary dimension}

\author{\Large S. Lawrencenko, M.N. Vyalyi, L.V. Zgonnik}
%\date{August 24, 2012}
\maketitle

{\it \large Dedicated to the memory of our colleague

and friend, Dan Archdeacon.}
\medskip

\begin{abstract}
\noindent This paper is a collection of thoughts and observations, being partly a review and partly a report of current research, on recent work in various aspects of Gr\"unbaum colorings, their existence and usage. In particular, one of the most striking significances of Gr\"unbaum's Conjecture in the 2-dimensional case is its equivalence to the  4-Color Theorem.
The notion of Gr\"unbaum coloring is extended from the 2-dimensional case to the case of arbitrary finite hyper-dimensions.
\end{abstract}



\indent \indent 05B05, 52C20, 52C22,  05B07, 57M20, 57M15 (Secondary).

\section{Introduction}
 
\noindent This paper is a collection of thoughts and observations, being partly a review and partly a report of current research, on recent work in various aspects of Gr\"unbaum colorings, their existence and usage. In particular, we recall some of Dan Archdeacon's thoughts on this subject, in particular, his reconsidering Gr\"unbaum colorings in dual form. In Section 2, we state our basic result. In Sections 3 and 4, we address the 2-dimensional case; in particular, one of the most striking significances of Gr\"unbaum's Conjecture (for the 2-sphere case) is its equivalence to the  {4-Color} Theorem \cite{AH, AHK}. In Section 5, we suggest how to construct new facet 2-colorable triangulations from old ones and expose infinite series of facet 2-colorable triangulations of the $n$-sphere for each $n$. In Section 6, we describe four applications of Gr\"unbaum colorability and facet 2-colorability in the case of general dimension. For notation and graph theory terminology we generally follow \cite{H, JT}.

A {\it vertex} [resp. {\it edge}] ({\it proper}) {\it $k$-coloring} of a triangulation $T$ is a surjection of the vertex set $V(T)$ [resp. edge set $E(T)$] onto a set of $k$ distinct colors such that the images of adjacent vertices [resp. edges] are different. 
The {\it vertex} [resp. {\it edge}] {\it chromatic numbers} of $T$ are defined to be the smallest values of $k$ possible to obtain corresponding $k$-colorings, and are denoted by $\chi(T)$ [resp.  $\chi^\prime(T)$]. The numbers $\chi(T)$ and $\chi^\prime(T)$  are also called the vertex and edge chromatic numbers of the graph $G(T)$  itself and denoted as $\chi(G(T))$  and $\chi^\prime(G(T))$,  respectively.

The term ``Gr\"unbaum coloring'' was introduced by Dan Archdeacon \cite{A-2}. The notion of {\it Gr\"unbaum coloring} was first introduced by Branko Gr\"unbaum \cite{G} as an assignment of three fixed colors to the edges of a triangulation $T$  of a closed compact 2-manifold so that the edges bounding every face of $T$ are assigned three different colors. Thus, a Gr\"unbaum coloring is a peculiar way of coloring edges which is different from the usual (proper) edge coloring. However, each Gr\"unbaum coloring of $T$ induces a Tait coloring of the graph $G^*(T)$  dual to the graph $G(T)$ of $T$. A {\it Tait coloring} is a usual edge 3-coloring of a 3-regular graph. 

By Vizing's Theorem \cite{V}, $\chi^\prime(G^*(T)) \in \{3, 4 \}$ and so four colors will certainly suffice for coloring the edges of any triangulation $T$ of any closed 
2-manifold so that the three edges have distinct colors around each face of $T$, which is mentioned in \cite{ JT, LM}. 
There are other types of cyclic colorations; for instance, one may color the faces of $T$ so that the face colors are all distinct around each vertex of $T$; see \cite{CL}.

\par\medskip 
\noindent {\bf Gr\"unbaum's Conjecture} (\cite{G}, 1969)~Every triangulation $T$  of an orientable closed 2-manifold is Gr\"unbaum colorable.
\par\medskip    

Gr\"unbaum believed that his Conjecture was true (personal communication to Dan Archdeacon), but Dan believed it was false \cite{A-2}. Gr\"unbaum's Conjecture stood for 40 years, until Martin Kochol \cite{Koch} constructed infinite families of counterexamples on orientable 2-manifolds with genus $g$ for all $g \geqslant 5.$ For the 2-sphere, Gr\"unbaum's Conjecture is true since it is equivalent to the 4-Color Theorem (Section 4). For the 2-torus, Gr\"unbaum's Conjecture has been proved by Albertson, Alpert, Belcastro, and Haas \cite{AABH} in the case where $\chi(T) \neq 5$.  The Conjecture is still open for the case $g =1$, $\chi(T) = 5$,  as well as for the cases $g \in \{2, 3, 4 \}$.

It has been known for a long time that Gr\"unbaum's Conjecture is false if extended to the nonorientable case (e.g., see \cite{A-2, JT}).
For instance, the minimal triangulation $T_{\min}$  of the projective plane by the complete 6-graph $G=K_6$  has the Petersen Graph as its dual graph $G^*(T_{\min})$ \cite{P} which has edge chromatic number equal to 4; see \cite{H, JT, LP, P}. The Petersen graph appeared to be the first known ``snark'' (Section 4).

The notion of Gr\"unbaum coloring is generalized in \cite{LM} for an arbitrary face size 
$\ell \geqslant 3$. In this paper we consider only triangulations ($\ell = 3$ in the case of dimension $n=2$) but generalize in another direction: we extend the notion of Gr\"unbaum coloring from the 2-dimensional case to the general case of arbitrary (finite) dimension $n\geqslant 1$.

By an {\it $n$-manifold} we always mean an $n$-dimensional compact closed manifold. Let $T$  be a triangulation of an $n$-manifold. Thinking of $T$  as the boundary complex of some $(n+1)$-polytope, we call the $n$-simplexes of $T$  {\it facets}. The number of $(n-1)$-faces in each facet of $T$  is equal to $(n+1)$.  We say that $T$  is {\it Gr\"unbaum hyper-colorable} if the set of $(n-1)$-simplexes of $T$  can be $(n+1)$-colored so that each facet of $T$ contains all the $(n+1)$  colors in its boundary complex. 

We say that $T$ is {\it facet 2-colorable} if all the facets of $T$ can be 2-colored, say black and white, so that any two facets that have a common $(n-1)$-face are colored differently. In the case of the smallest dimension $n=1$, the existence of a Gr\"unbaum hyper-coloring (that is, vertex coloring in this specific case) is equivalent to the existence of a facet 2-coloring (that is, edge coloring); either coloring exists if and only if $T$ is a 1-dimensional simple cycle of even length.

\section{Basic observation}

\noindent Following \cite{DH}, we say that two facets of $T$ are {\it $(n-1)$-adjacent} if they are incident with a common $(n-1)$-face. The $(n, n-1)$-adjacency graph of a complex $T$, written $A_{n, n-1}(T)$,
is the graph whose vertices are the facets of $T$,  two vertices of $A_{n, n-1}(T)$  being adjacent if the corresponding facets are $(n-1)$-adjacent. If $T$ is 1-dimensional (that is, a graph), then  $A_{1, 0}(T)$ is exactly the line graph of $T$ and is isomorphic to $T$. In \cite{DH}, the construction is more general; the authors introduce the notion of an $(n,m)$-adjacency graph for an arbitrary $n$-complex $K$, but in this paper we restrict attention to the case in which $K$ ($=T$) is a triangulation of a closed $n$-manifold and call $A_{n, n-1}(T)$ the {\it facet-adjacency graph} of $T$, abbreviating the notation to $A(T)$.

Since we assume that $T$ is a simplicial $n$-complex, it follows that $A(T)$  is a simple graph, that is, $A(T)$ does not have loops and does not have parallel edges. Gr\"unbaum coloring was originally introduced in the case of dimension $n=2$, in which case $A(T)$ is the dual 3-regular graph.

\begin{theorem}
Given a triangulation  $T$  of an arbitrary closed $n$-manifold ($n \geqslant 1$), then facet 2-colorability of  $T$  entails  Gr\"unbaum hyper-colorability of  $T$. 
\end{theorem}

\begin{proof}
The proof is an application of a classical theorem of K\"onig \cite{Ko, JT, LP} in the right place at the right time. As observed in \cite{LM}, K\"onig's Theorem entails that each bipartite $\ell\mbox{-}$regular graph is edge $\ell$-colorable. Since a graph is bipartite if and only if it is vertex 2-colorable, it follows that if $T$ is facet 2-colorable, then $A(T)$  is an $(n+1)$-regular bipartite graph and by K\"onig's Theorem admits an edge $(n+1)$-coloring which naturally induces a Gr\"unbaum hyper-coloring of $T$. 
\end{proof}

The 2-dimensional version ($n=2$) of Theorem~1 was established in \cite {LM} in a more general setting of $\ell\mbox{-}$angulations. 

\section{Complete graphs}

\noindent Dan Archdeacon was very interested in Gr\"unbaum colorings; in particular, he asked \cite{A-2} if Gr\"unbaum's Conjecture was true for complete graphs. 
It is a conjecture of Mohar \cite{M} that the answer to this question is yes.
As shown in \cite{LM}, a combination of known results by Ringel \cite{R}, Youngs \cite{Y}, Grannell, Griggs, and \v{S}ir\'a\v{n} \cite{GGS}, Grannell and Korzhik \cite{GKo} on facet (2-face) 2-colorability along with Theorem~1 leads to the following corollary. 

\begin{corollary}
For each $p\equiv 3$ or $7$ {\rm (mod 12)},  there exists a Gr\"unbaum colorable orientable 2-dimensional triangulation by the complete graph $K_p$. Furthermore, for each  $p\equiv 1$ or $3$ {\rm (mod 6)}, $p \geqslant 9$,  there exists a Gr\"unbaum colorable nonorientable triangulation by $K_p$.
\end{corollary}

Corollary 1 establishes the existence of Gr\"unbaum colorable triangulations on orientable and nonorientable surfaces by complete graphs $K_p$  for at least half of the residue classes in the spectrum of possible values of $p$.

As noted in \cite{LM}, the property of being facet 2-colorable is an isomorphism invariant of a triangulation. (Two triangulations are called isomorphic provided that there exists a bijection between their vertex sets that preserves faces.) The first examples of non-isomorphic orientable triangulations by the same complete graph were constructed by Youngs \cite{Y}. In all those examples, the non-isomorphism is due to the fact that one of the triangulations in each pair is facet $2\mbox{-}$colorable while the other is not (see review \cite{GG} for details). After a quarter of a century, in \cite{LNW}, there was constructed an example of {\sl more than two} pairwise non-isomorphic orientable triangulations (all by the same complete graph), namely: There were constructed {\sl three} such triangulations, only one of which is facet $2\mbox{-}$colorable; thus, \cite{LNW} provided historically the first example of two non-isomorphic, {\sl facet non-2-colorable} triangulations by the same complete graph. In 2000, it was shown by Bonnington, Grannell, Griggs, and \v{S}ir\'a\v{n} \cite{BGGS} (also see \cite{GG}) that the number of non-isomorphic orientable triangulations by the graph $K_p$  actually grows very rapidly as $p \to \infty$~ even within the restricted class of face $2\mbox{-}$colorable triangulations; for instance, when $p\equiv 7$ or $19$ (mod $36$), that number is not less than $2^{p^2/54-o(p^2)}$. 

As also noted in \cite{LM}, another corollary is obtained by a combination of related results of Ringel and Youngs \cite{RY}, Grannell, Griggs, and Knor \cite{GGK} (also see \cite{GG}) on facet 2-colorability along with Theorem~1: {\it For each $p \geqslant 2$,  all orientable triangulations by the complete tripartite graph $K_{p,p,p}$  are Gr\"unbaum colorable.} However, a roundabout method sometimes produces a trivial result. A direct simple proof was given by Dan Archdeacon, which produces a stronger result: {\it All orientable and nonorientable triangulations by any tripartite graphs (complete or not) are Gr\"unbaum colorable.} Dan's proof \cite{A}: If the vertex parts are $A$, $B$, $C$, then it suffices to color the edges between $A$  and $B$  red, those between $B$  and $C$  blue, and those between $A$  and $C$  green.

Archdeacon's result mentioned in the preceding paragraph is not a sort of ``on spherical chickens in a vacuum'' as it may seem at first sight. As shown in \cite{GGKS} (see also \cite{GG}) in the case $p$  is prime, there already exist at least $(p-2)!/(6p)$  non-isomorphic orientable triangulations by $K_{p,p,p}$.  Furthermore, \cite{GKn} provides improved bounds on the number of such triangulations; for instance, when $p\equiv 6$ or $30$ (mod $36)$,  there exist at least $p^{p^2/144-o(p^2)}$  non-isomorphic orientable triangulations by $K_{p,p,p}$.

\section{The 4-Color Theorem and snarks}

\begin{theorem} [\bf Tait]
In the 2-sphere case Gr\"unbaum's Conjecture is equivalent to the 4-Color Theorem. 
\end{theorem}

To prove that  Gr\"unbaum's Conjecture implies the 4-Color Theorem, Tait observed \cite{T, BLW} that the 4-Color Theorem in the dual formulation is equivalent to the statement that every simple planar 3-regular bridgeless graph has a Tait coloring; see also \cite{BM, GS, SK}. The converse implication easily follows by the following folklore result.

\begin{lemma} [\bf folklore, see \cite{AABH}]
Each vertex 4-colorable triangulation $T$ of any 2-manifold is Gr\"unbaum colorable.
\end{lemma}

\begin{proof}
The proof is very short and disarmingly simple. We reproduce it here, inspired by Archdeacon's proof of his proposition in Section 3. 
The vertex 4-coloring enables partitioning the vertex set of $T$  into four disjoint vertex color classes
denoted by $A$, $B$, $C$, and $D$.  We then color the edges between $A$  and $B$  and also the ones between $C$  and $D$  red, those between $A$  and $C$  and also between $B$  and $D$  blue, and those between $A$  and $D$  and between $B$  and $C$ green, and we are done because any cycle of $T$ with length 3 has edges of all three colors.
\end{proof}

In the remainder of this section let $G^*$ denote any simple (not necessarily planar) 3-regular bridgeless graph. (The reader is urged not to be misled by the $*$ in the superscript of $G$; the asterisk is used in this section only for consistency with Section 1 but without emphasis on ``dual graph''.) It is a conjecture of Tutte \cite{Tutte1, Tutte2} that if $G^*$ has no subgraph which is a subdivision of the Petersen Graph, then $G^*$ is Tait colorable. Some partial results towards Tutte's Conjecture have been established, e.g. \cite{Seymour-2, E}.

As observed by Seymour \cite{Seymour}, a conjecture of Gr\"otzsch implies that if $G^*$ has a vertex $v$ so that $G^*-v$ is planar, then $G^*$ is Tait colorable. 
This statement is less general than Tutte's Conjecture because the graph resulting from the Petersen Graph by deleting any one vertex is nonplanar; see \cite{JT}.  Following Gardner \cite{Gardner}, if $G^*$ is not Tait colorable, $G^*$ is called a {\it snark}. By Vizing's Theorem,  $G^*$ is a snark if and only if $\chi^{\prime} (G^{*}) = 4$.

The term ``snark'' is motivated by the rarity of snarks (\cite{Descartes}). In fact, the first infinite family of snarks was constructed by Isaacs \cite{Isaacs}. 
Later on, several more such families were found (see the survey \cite{Watkins}).
However, the Tait colorability problem is NP-complete (Holyer \cite{Holyer}), and so the snark recognition is coNP-complete. Thus, it seems pretty unlikely that a reasonably simple characterization
of snarks would ever be available.

Dan Archdeacon was very interested in snarks; see \cite{A-2}. An embedding of a graph on a 2-manifold is called {\it polyhedral} if any two distinct facets share at most a vertex or edge in their boundaries. It is an observation of Dan that Gr\"unbaum's Conjecture is equivalent to the following statement: No snark has a polyhedral embedding. 
Dan, together with his friends Paul Bonnington and Jozef \v{S}ir\'a\v{n}, found and examined many of the known snarks and showed that they had no polyhedral embedding. Their search was helped by searching for so-called {\it poly-killers}, that is, minimal subgraphs with the property that they cannot occur in any graph with a polyhedral embedding. 

As a possible direction towards partially proving Gr\"unbaum's Conjecture, Dan conjectured \cite{A-2} that a Gr\"unbaum coloring exists for those triangulations with every noncontractible cycle of length at least $l$ for some large enough $l$ depending on the 2-manifold. In connection with this we mention a conjecture by Jaeger and Swart \cite{JS} that every snark has a cycle with length at most 6. This conjecture was disproved by Kochol (see \cite[p. 198]{JT}), who constructed snarks with arbitrarily long shortest cycles.

In conclusion of this section we draw attention to even (or Eulerian) triangulations of 2-manifolds. By Theorem 1, every even triangulation of the 2-sphere is Gr\"unbaum colorable. 
This has been generalized for the 2-torus, projective plane and Klein bottle in \cite{KMMNNOV}, which apparently required the combined efforts of seven co-authors.
Moreover, they have proved that every triangulation of a fixed 2-manifold with sufficiently large representativity (depending on the 2-manifold) is Gr\"unbaum colorable, thus having proven Dan Archdeacon's conjecture (mentioned in the preceding paragraph) for all even triangulations.

\section{Generation of facet 2-colorable triangulations}

\noindent Little is known about facet 2-colorable triangulations of $n$-manifolds. Aside from the authors of this paper, the facet 2-colorability of the first barycentric subdivision of any orientable closed triangulation was also observed by Schleimer \cite {S-2}.

\begin{corollary}
The first barycentric subdivision of any triangulation of any closed orientable $n$-manifold is a facet 2-colorable, and thereby Gr\"unbaum hyper-colorable, triangulation. 
\end{corollary}

\begin{corollary}
For any closed, orientable $n$-manifold $M^n$, there exists a Gr\"unbaum hyper-colorable triangulation of $M^n$.
\end{corollary}

Thus, one of the significances of facet 2-colorable triangulations is that they provide a sufficient basis for studying topological properties of closed orientable $n$-manifolds (see also Section 6.4).

In the general dimension case, besides the barycentric subdivision, there are two other operators to generate new facet 2-colorable triangulations from old ones. Specifically, in the 2-dimensional case, the facet 2-colorable triangulations of 2-manifolds by complete graphs established in Corollary 1 can serve as a basis for generating more facet 2-colorable 2-dimensional triangulations. Let $T$ and $R$ be facet 2-colorable triangulations of given  closed $n$-manifolds $M$ and $L$, respectively.

Firstly, the binary {\it glue operator} merges a black facet of $T$ with a white facet of $R$, and then deletes the two facets together with the simplexes on the boundary complex of one of the deleted facets. Clearly, the glue operator produces a facet 2-colorable triangulation of the connected sum $M \# L$.

Secondly, another generation method uses the unary operator of {\it bipyramidal crowning.} This method works as follows. Create a new vertex, $N$. 
The ``northern cone'' over $T$ is the $(n+1)$-dimensional pyramid with base $T$ and peak $N$ which consists of all simplexes of $T$ as well as all the simplexes determined by $N$ and the simplexes of $T$. Similarly, we construct the ``southern cone'' over $T$ with peak $S$. Then we color each new facet (that is, $(n+1)$-simplex) incident with $N$ the same color as the basic $n$-face, and reverse the color for the antipodal facet incident with $S$. Then the union of the two so colored cones is a facet 2-colored triangulation of a closed $(n+1)$-manifold. 
In particular, if the base $T$  is a facet 2-colored triangulation of the $n$-sphere, the bipyramidal crown is a facet 2-colored triangulation of the 
$(n+1)$-sphere. As a more specific and geometrically nice example, the boundary complex of the $(n+1)$-cross-polytope has as 1-skeleton the complete $(n+1)$-partite graph $K_{2, \ldots, 2}$  ($n+1$  twos) and is a facet 2-colorable triangulation of the $n$-sphere.

\section{Applications of Gr\"unbaum hyper-coloring \\ and facet 2-coloring}

\medskip
\noindent {\bf 6.1. A model of the polymerized membrane.} Among the five Platonic solids in $\mathbb{R}^3$, only the tetrahedron, octahedron, and icosahedron are made of equilateral triangles. Among these three triangulations, only the octahedron is 2-face 2-colorable, which yields some desirable properties such as Gr\"unbaum colorability. Those favorable properties in turn yield the 3-colorability of the links of the regular triangular lattice as used by Bowick, Di Francesco, Golinelli, and Guitter \cite{BFGG} to model the polymerised membrane. As shown by Gottlieb and Shelton \cite{GS}, there are precisely two nonisomorphic Gr\"unbaum colorings of the octahedron.

\medskip
\noindent {\bf 6.2. From a triangulation to a quadrangulation.} If we remove all edges assigned one and the same color from a Gr\"unbaum colored 2-dimensional triangulation, we comfortably get a quadrangulation of the same 2-manifold. Furthermore, if we remove all the 2-simplexes with the same color from a Gr\"unbaum hyper-colored  triangulation of a closed 3-manifold $M^3$, we obtain a tessellation of $M^3$  that divides it into a series of  (topological) triangular bipyramids. A similar technique can be used in higher dimensions as well.  

\medskip
\noindent {\bf 6.3. Isohedral scalene realization of triangulations.} Our next application is related to the DGP (Distance Geometry Problem) \cite{LLMM}. 
We consider the following version of the DGP, still open: 
Decide whether a given triangulation $T$ of a closed $n$-manifold can be realized geometrically as an {\it isohedral scalene polyhedron} in $\mathbb {R}^{n+1}$, that is, so that (i)~all facets of $T$ are realized by congruent geometric $n$-simplexes and (ii)~the $(n-1)$-volumes of the $(n-1)$-faces of each facet are pairwise different. 
The importance of the isohedral scalene polyhedra is that their symmetry groups act freely on their facet sets. 
Perhaps the most well-known example of an isohedral scalene polyhedron is the boundary complex of a rhombic disphenoid in $\mathbb{R}^3$; it has 4 congruent scalene (2-dimensional) facets and has three pairwise different edge lengths. 

The combinatorial idea behind our approach here is that we first ``color'' the $(n-1)$-simplexes of $T$  with $n+1$  given values (fixed real numbers) instead of colors.
More specifically, we assign the $(n-1)$-simplexes of $T$ their prescribed $(n-1)$-volumes pairwise different within each facet, which is possible if and only if $T$ is Gr\"unbaum hyper-colorable. Geometrically, we need to 
assign the $(n-1)$-volumes coherently; for instance, for $n=2$, the $(n-1)$-volumes (lengths) should satisfy the triangle inequality.  

\medskip
\noindent {\bf 6.4. The 3-sphere recognition problem.} This problem is still open; see Question 12.1 \cite[p. 533]{BP}.
Casson's version \cite{Casson} of Thompson's Theorem \cite{Thompson} which relies heavily on work of Rubinstein \cite {R-2} (see also \cite{S}) states that there is an exponential time algorithm which, given a triangulation $T$, decides whether or not the carrier $|T|$  is homeomorphic to the 3-sphere. Furthermore, the 3-sphere recognition problem lies in the complexity class NP, that is, if $T$ is a triangulation of the 3-sphere then there is a polynomial sized algorithm---Schleimer's Algorithm \cite{S}---to prove this fact.

By Corollaries 2 and 3, the 3-sphere recognition problem in the orientable case has the same complexity as the problem restricted to the facet (3-face) 2-colorable 3-dimensional triangulations.
Thus, we may assume that the triangulation in question is facet 2-colorable
when applying Schleimer's Algorithm. 
It could be tempting to simplify Schleimer's Algorithm,
for example, by using the barycentric subdivision. 
However, Schleimer has himself pointed out \cite{S-2} that in the crushing phase of the algorithm, the facet 2-colorability will almost certainly be abandoned; also, the complexity function (number of facets) needs to be strictly decreasing once the process starts; so we cannot automatically recover the facet 2-coloring.
Finally, the bottom line is that Hass and Kuperberg \cite{HK} have announced the result that
the 3-sphere recognition problem lies in the complexity class coNP, modulo the generalized Riemann hypothesis, which means that NP-completeness of the problem is unlikely to be the case.

\section*{Acknowledgments}

The first author wishes to acknowledge email conversations with Saul Schleimer and is grateful for the enthusiastic discussions in person with Joachim Hyam Rubinstein on the sidelines of the First Pacific Rim Conference on Mathematics (January 19--23, 1998, City University of Hong Kong), on the subject of the 3-sphere recognition problem.

\bibliographystyle{model1-num-names}
\bibliography{<your-bib-database>}

\par\medskip 
\par\medskip 
\par\medskip 

\noindent \textsc{S.\,\,Lawrencenko} \\
\noindent Russian State University of Tourism and Service,\\
\noindent Institute of Tourism and Service (Lyubertsy),\\
\noindent Lyubertsy, Moscow Region, Russia,\\
\noindent e-mail: \url{lawrencenko@hotmail.com}\\

\noindent \textsc{M.\,N.\,\,Vyalyi} \\
\noindent National Research University Higher School of Economics,\\
\noindent Moscow, Russia,\\
\noindent e-mail: \url{vyalyi@gmail.com}\\

\noindent \textsc{L.\,V.\,\,Zgonnik} \\
\noindent Russian State University of Tourism and Service,\\
\noindent Institute of Tourism and Service (Lyubertsy),\\
\noindent Lyubertsy, Moscow Region, Russia,\\
\noindent e-mail: \url{mila.zgonnik1@yandex.ru}\\

\end{document}